\documentclass[12pt]{article}

\usepackage{times,amsfonts,amsmath,amstext,amsbsy,amssymb,
  amsopn,amsthm,upref,eucal, amscd}
\usepackage{fullpage}
\usepackage[colorlinks=true]{hyperref}
\usepackage[T1]{fontenc}
\usepackage[backgroundcolor=white]{todonotes}
\usepackage{multirow}
\usepackage{color}

\usepackage{graphicx}

\newtheorem{theorem}{Theorem}[section]
\newtheorem{lemma}[theorem]{Lemma}
\newtheorem{corollary}[theorem]{Corollary}
\newtheorem{proposition}[theorem]{Proposition}

\numberwithin{equation}{section}

\theoremstyle{definition}
\newtheorem{definition}[theorem]{Definition}
\newtheorem{example}[theorem]{Example}

\newtheorem{remark}[theorem]{Remark}

\makeatletter
\newcommand{\subjclass}[2][1991]{%
  \let\@oldtitle\@title%
  \gdef\@title{\@oldtitle\footnotetext{#1 \emph{Mathematics subject classification.} #2}}%
}

\def\leq{\leqslant }
\def\geq{\geqslant}

\usepackage{accents}

\begin{document}

\title{Sierpi\'nski  fractals and the  dimension of their Laplacian spectrum}

\author{Mark Pollicott and Julia Slipantschuk\footnote{The authors were partly supported by ERC-Advanced Grant 833802-Resonances.}}

\subjclass[2020]{28A80, 37F35, 37C30, 65D05}


\newcommand{\Addresses}{{
  \bigskip
  \footnotesize

  M.~Pollicott, \text{Department of Mathematics, University of Warwick, Coventry, CV4 7AL, UK.}\par\nopagebreak
  \textit{E-mail}: \texttt{masdbl@warwick.ac.uk}

  \medskip

  J.~Slipantschuk, \text{Department of Mathematics, University of Warwick, Coventry, CV4 7AL, UK.}\par\nopagebreak
  \textit{E-mail}: \texttt{julia.slipantschuk@warwick.ac.uk}
}}

\maketitle

\center{Dedicated to K\'aroly Simon on the occasion of his 60+1st birthday}

\abstract{We estabish rigorous estimates for the Hausdorff dimension of the spectra
of Laplacians associated to Sierpi\'nski lattices and infinite Sierpi\'nski gaskets and other
post-critically finite self-similar sets.}

\section{Introduction}

The study of the Laplacian on manifolds has been a very successful area of mathematical analysis for over a century, combining ideas from topology, geometry, probability theory and harmonic analysis. A comparatively new development is the  theory of a  Laplacian for certain types of naturally occurring fractals, see \cite{strichartz-notices,sabot2,teplaev,fs,shima,kigami,rammal}, to name but a few. A particularly well-known example is the following famous set.

\begin{definition}
The \emph{Sierpi\'nski triangle}
$\mathcal T \subset \mathbb R^2$ (see Figure \ref{fig:sierpinski}(a)) is
the smallest non-empty compact set\footnote{In the literature, this set is also often referred to as the \emph{Sierpi\'nski gasket}, and denoted $SG_2$.}
such that
$\bigcup_{i=1}^3 T_i (\mathcal T) = \mathcal T$ where $T_1, T_2, T_3\colon \mathbb R^2 \to \mathbb R^2$ are the affine maps
$$
\begin{aligned}
T_1(x,y) &= \left(\frac{x}{2}, \frac{y}{2}\right) \qquad
T_2(x,y) = \left(\frac{x}{2} + \frac{1}{2}, \frac{y}{2}\right)\cr
T_3(x,y) &= \left(\frac{x}{2} + \frac{1}{4}, \frac{y}{2} + \frac{\sqrt{3}}{4}\right).\cr
\end{aligned}
$$
\end{definition}

\begin{figure}[h!]
\centerline{\includegraphics[height=3.5cm, angle=-0]{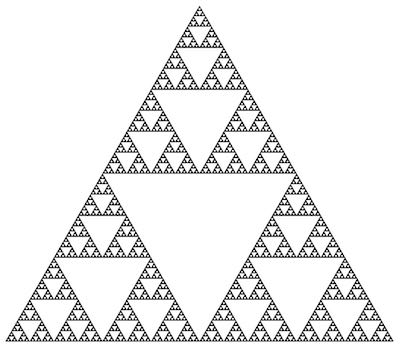}
\hskip 0.5cm
\includegraphics[height=3.5cm, angle=-0]{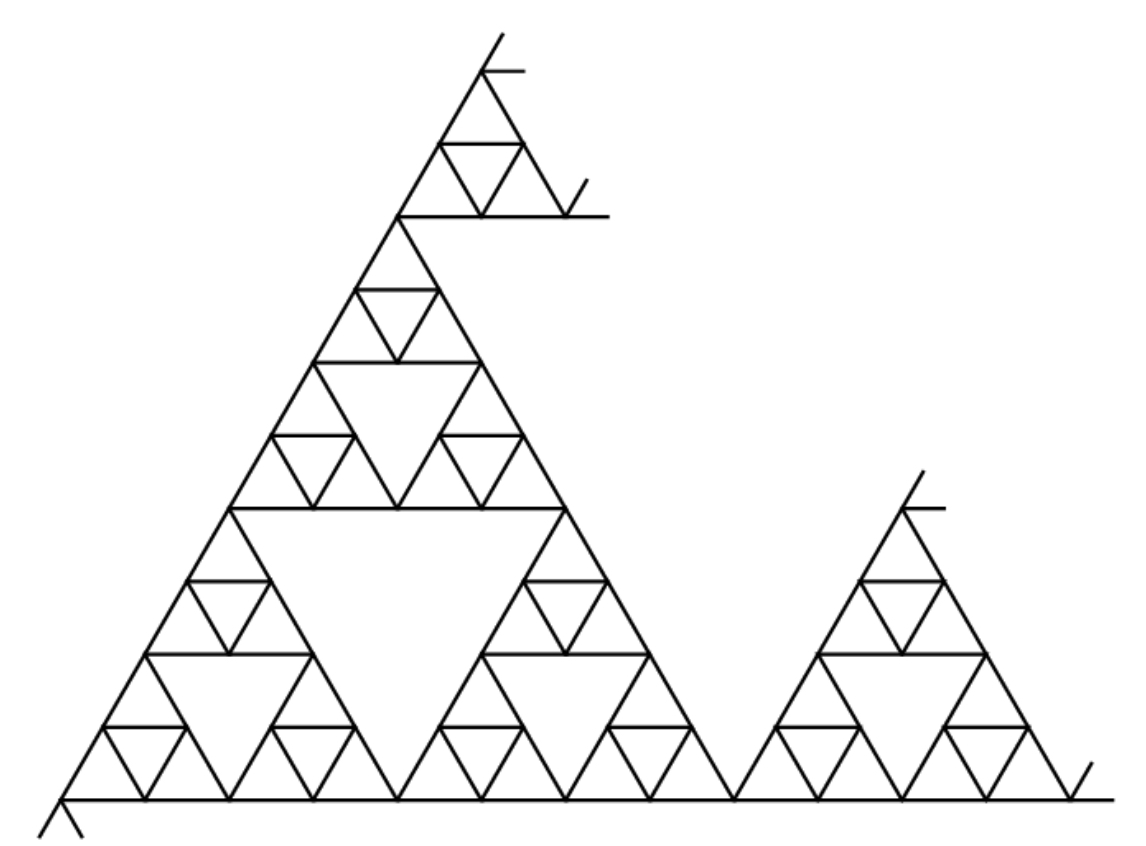}
\hskip 0.5cm
\includegraphics[height=3.5cm, angle=-0]{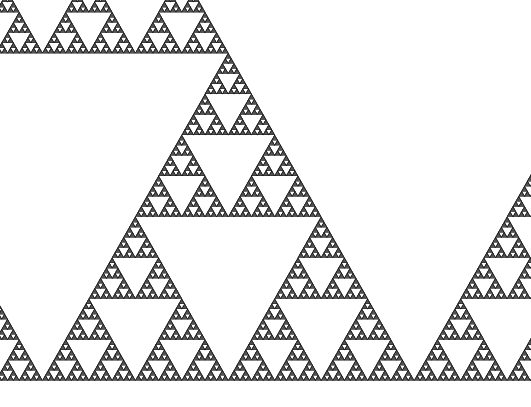}
}
\caption{(a) The standard Sierpi\'nski triangle $\mathcal T$;
(b) The Sierpi\'nski lattice ${\mathcal L}$; and
(c) The infinite Sierpi\'nski triangle $\mathcal T^\infty$.}
\label{fig:sierpinski}
\end{figure}

A second object which will play a role is the following infinite graph.

\begin{definition}
Let $V_0 = \{(0,0), (1,0), (\frac{1}{2}, \frac{\sqrt{3}}{2})\}$
be the vertices of $\mathcal T$ and define
$V_n = \bigcup_{i=1}^3 T_i(V_{n-1})$. Further, fix a sequence
$\omega = (\omega_n)_{n \in \mathbb{N}} \subset \{1, 2, 3\}^{\mathbb N}$, and let
$$V^\infty = \bigcup_{i=1}^\infty V^n \text{ with } V^n  = T_{\omega_1}^{-1} \circ \cdots \circ T_{\omega_n}^{-1} (V_n),$$
where we use the inverses
$$
\begin{aligned}
T_1^{-1}(x,y) &= \left(2x, 2y \right) \qquad
T_2^{-1}(x,y) = \left(2x -1,2y\right)\cr
T_3^{-1}(x,y) &= \left(2x-\frac{1}{2}, 2y - \frac{\sqrt{3}}{2}\right).\cr
\end{aligned}
$$
The definition of $V^\infty$ depends on the choice of $\omega$,
however as will be explained below, the relevant results
do not, allowing us to omit the dependence in our notation.
The points in $V^\infty$ correspond to the vertices of an infinite graph $\mathcal L$
called a \emph{Sierpi\'nski lattice}
 for which the  edges correspond to pairs of vertices $(v,v')$,
with $v,v' \in V^\infty$ such that  $\|v-v'\|_2=1$
(see Figure \ref{fig:sierpinski}(b)).
Equivalently, $\mathcal L$ has an edge $(v, v')$ if and only if
$$v, v' \in T_{\omega_1}^{-1} \circ\cdots \circ T_{\omega_n}^{-1} \circ
T_{i_n} \circ \cdots \circ T_{i_1}(V_0)$$
for some $i_1, \ldots, i_n \in \{1, 2, 3\}$, $n \geq 0$.
\end{definition}

\medskip
Finally, we will also be interested in
infinite Sierpi\'nski gaskets, which can be defined similarly to Sierpi\'nski lattices as follows.

\begin{definition}
For a fixed sequence $\omega = (\omega_n)_{n\in \mathbb N}$, we define
an {\it infinite Sierpi\'nski gasket} to be the unbounded set $\mathcal T^\infty$ given by
$$\mathcal T^\infty  = \bigcup_{n=0}^{\infty} \mathcal T^n, \text{ with } \mathcal T^n = T_{\omega_1}^{-1} \circ \cdots \circ T_{\omega_n}^{-1} (\mathcal T),
$$
which
is a countable union of copies of the standard Sierpi\'nski triangle $\mathcal T$
(see Figure \ref{fig:sierpinski}(c)). As for Sierpi\'nski lattices, the definition of $\mathcal T^\infty$
depends on the choice of $\omega$, but we omit this dependence in our notation as the
cited results hold independently of it.
\end{definition}

\medskip

The maps $T_1$, $T_2$ and $T_3$ are similarities on $\mathbb R^2$ with respect to the Euclidean norm,
and more precisely
$$\|T_i(x_1,y_1) - T_i(x_2,y_2)\|_2 = \frac{1}{2} \|(x_1,y_1) - (x_2,y_2)\|_2$$ for $(x_1,y_1), (x_2,y_2) \in \mathbb R^2$ and $i=1,2,3$,
 and thus by  Moran's theorem  the Hausdorff dimension of
$\mathcal T$  has the  explicit value $\dim_H(\mathcal T) = \frac{\log 3}{\log 2}$  \cite{falconer}.
We can easily give the Hausdorff dimensions of the other  spaces.
It is clear that $\dim_H(\mathcal L) =1$ and
since an infinite Sierpi\'nski gasket $\mathcal T^\infty$ consists of countably many copies of $\mathcal T$ it follows that we also have
$\dim_H(\mathcal T^\infty) = \frac{\log 3}{\log 2}$.

In this note we are concerned  with other  fractal sets  closely associated to
the infinite Sierpi\'nski gasket $\mathcal T^\infty$ and the  Sierpi\'nski  lattice $\mathcal L$, for which
 the Hausdorff dimensions are significantly more difficult to compute.

\bigskip
In \S \ref{sec:spec_laplacians} we will describe how to
associate to $\mathcal T$ a Laplacian $\Delta_{\mathcal T}$ which is a linear operator
defined on   suitable functions $f\colon \mathcal T \to \mathbb R$.
An eigenvalue $\lambda \geq  0$  for  $-\Delta_{\mathcal T}$ on the Sierpi\'nski triangle is then   a solution to the basic identity
$$\Delta_{\mathcal T} f + \lambda f =0.$$
The spectrum  $\sigma (-\Delta_{\mathcal T}) \subset \mathbb R^+$  of $-\Delta_{\mathcal T}$ is a countable set of eigenvalues.
In particular, its Hausdorff dimension
satisfies $\dim_H(\sigma(-\Delta_{\mathcal T}))=0$.
A nice account of this theory appears in the survey note of
Strichartz \cite{strichartz-notices} and his book \cite{strichartz-book}.

By contrast, in the case of the infinite Sierpi\'nski gasket and the Sierpi\'nski lattice there are associated Laplacians, denoted $\Delta_{\mathcal T^\infty}$ and $\Delta_{\mathcal L}$, respectively, with spectra $\sigma(-\Delta_{\mathcal T^\infty}) \subset \mathbb R^+$ and $\sigma(-\Delta_{\mathcal L}) \subset \mathbb R^+$ which are significantly more complicated.
In particular, their Hausdorff dimensions are non-zero and
    therefore their numerical values are of potential interest.   However, unlike the case of the
    dimensions of the original sets  $\mathcal T^\infty$ and $\mathcal L$, there is  no clear explicit form for this quantity.
    Fortunately, using thermodynamic   methods we can estimate the Hausdorff dimension\footnote{In
    this case the Hausdorff dimension equals the Box counting dimension, as will become apparent in the proof.}
 numerically to very high precision.

\begin{theorem}\label{thm:main}
The Hausdorff dimension of
 $\sigma(-\Delta_{\mathcal T^\infty})$ and $\sigma(-\Delta_{\mathcal L})$
 satisfy
$$
\dim_H(\sigma(-\Delta_{\mathcal T^\infty})) =
\dim_H\left(\sigma(-\Delta_{\mathcal L}) \right)=
0.55161856837246 \ldots
$$
\end{theorem}

A key point in our approach is that we have rigorous bounds, and
the value in the above theorem is accurate to the number of decimal places presented. We can actually estimate this Hausdorff dimension to far more decimal places. To illustrate this, in the final section we give an approximation to
$100$ decimal places.

It may not be  immediately obvious what practical information the numerical value of the Hausdorff dimension gives about the
sets $\mathcal T^\infty$ and $\mathcal L$ but it may have the potential to give an interesting numerical characteristic of the spectra.
Beyond pure fractal geometry, the spectra of Laplacians on fractals are also of practical interest, for instance
in the study of vibrations in heterogeneous and random media, or the design of so-called fractal antennas
\cite{C,HC}.

We briefly summarize the contents of this note.
In  \S \ref{sec:spec_laplacians} we describe some of the background for the Laplacian on the Sierpi\'nski graph.
 In particular, in \S \ref{sec:spec_dec} we recall the basic approach of {\it decimation} which allows
$\sigma(\Delta_{\mathcal T})$
 to be expressed in terms of a polynomial $R_{\mathcal T}(x)$.
 Although we are not directly interested in the zero-dimensional set  $\sigma(-\Delta_{\mathcal T})$, the
spectra of $\sigma(-\Delta_{\mathcal T^\infty})$ and $\sigma(-\Delta_{\mathcal L})$ actually contain  a  Cantor set $\mathcal J_{\mathcal{T}} \subset [0,5]$,
the so-called Julia set associated to the polynomial  $R_{\mathcal T}(x)$.

As one would expect, other related  constructions of fractal sets have similar spectral properties and their dimension can be
similarly studied.
In \S \ref{sec:related_results} we consider higher-dimensional Sierpi\'nski simplices, post-critically finite fractals, and
an analogous problem where we consider the spectrum of the Laplacian
on infinite graphs (e.g., the Sierpi\'nski graph and the Pascal graph).
In \S \ref{sec:dim_estimate} we recall the algorithm we used to estimate the dimension and describe its application.
This serves to both justify our estimates and also to use them as a way to illustrate a method with wider applications.

\section{Spectra of the Laplacians} \label{sec:spec_laplacians}

\subsection{Energy forms}
There are various approaches to defining the Laplacian $\Delta_{\mathcal T}$ on $\mathcal T$. We will use  one of the simplest ones, using energy forms.

Following Kigami \cite{kigami} the definition
of the spectrum of the Laplacian for the Sierpi\'nski gasket $\mathcal T$
involves a natural sequence of finite graphs $X_n$ with
$$
X_0 \subset X_1 \subset  X_2 \subset \cdots \subset \bigcup_{n} X_n
\subset
\overline{\bigcup_{n} X_n}
=:
\mathcal T,
$$
the first three of which
are  illustrated in Figure \ref{fig:sierpinski_steps}.
To this end, let
$$V_0 = \left\{ (0,0), (1,0), \left(\frac{1}{2},  \frac{\sqrt{3}}{2}\right)\right\}$$ be the three vertices of $X_0$.
The vertices of $X_n$ can  be defined iteratively to be the set of points satisfying
$$
V_n = T_1(V_{n-1}) \cup T_2(V_{n-1}) \cup T_3(V_{n-1})
\quad \hbox{ for }  n \geq 1.
$$

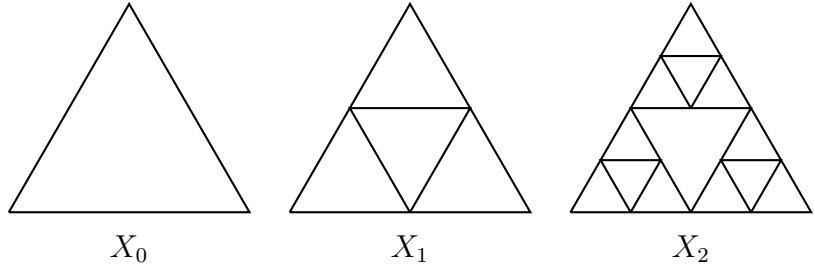
\begin{figure}
\centering
\begin{tikzpicture}[scale=0.32]
\draw[thick, black] (0,0) -- (10,0)--(5,8.66)--(0,0);
\node [below] at (5,-0.5){$X_0$};
\end{tikzpicture}
\hskip 0.4cm
\begin{tikzpicture}[scale=0.32]
\draw[thick, black] (0,0) -- (10,0)--(5,8.66)--(0,0);
\draw[thick, black] (5,0) -- (7.5,4.33)--(2.5, 4.33)--(5,0);
\node [below] at (5,-0.5){$X_1$};
\end{tikzpicture}
\hskip 0.4cm
\begin{tikzpicture}[scale=0.32]
\draw[thick, black] (0,0) -- (10,0)--(5,8.66)--(0,0);
\draw[thick, black] (5,0) -- (7.5,4.33)--(2.5, 4.33)--(5,0);
\draw[thick, black] (7.5,0) -- (8.75,2.165)--(6.25, 2.165)--(7.5,0);
\draw[thick, black] (2.5,0) -- (3.75,2.165)--(1.25, 2.165)--(2.5,0);
\draw[thick, black] (5,4.33) -- (6.25, 6.495)--(3.75, 6.495)--(5,4.33);
\node [below] at (5,-0.5){$X_2$};
\end{tikzpicture}
\caption{The first three graphs for the Sierpi\'nski triangle.}
\label{fig:sierpinski_steps}
\end{figure}

We denote by $\ell^2(V_n)$ (for $n \geq 0$) the real valued functions $f\colon V_n \to \mathbb R$ (where the $\ell^2$ notation is used for consistency with the infinite-dimensional case despite having no special significance for finite sets).

\begin{definition}
To each of the finite graphs $X_n$ ($n \geq 0$) we can associate bilinear forms
$\mathcal E_n\colon \ell^2(V_n) \times \ell^2(V_n) \to \mathbb R$ called \emph{self-similar energy forms} given by
\begin{equation}\label{eq:energyform}
\mathcal E_n(f, g) = c_n \sum_{x\sim_n y} (f(x) - f(y))(g(x) - g(y)),
\end{equation}
where $x,y \in V_n$ are vertices of $X_n$,  and $x\sim_n y$ denotes neighbouring edges in $X_n$.   In particular, $x \sim_n y$  precisely when
there exists $x',y' \in V_{n-1}$ and $i \in\{1,2,3\}$ such that $x=T_i(x')$ and $y=T_i(y')$.
 The value $c_n>0$ denotes a suitable scaling constant.
 With a slight abuse of notation, we also write $\mathcal E_n(f) := \mathcal E_n(f,f)$ for the
corresponding quadratic form $\ell^2(V_n) \to \mathbb R$.
\end{definition}            


To choose the values $c_n > 0$ (for $n \geq 0$)
we want the sequence  of bilinear forms $(\mathcal E_n)_{n=0}^\infty$
to be consistent by asking that for any $f_{n-1}\colon V_{n-1} \to \mathbb R$  (for $n \geq 1$) we have
  $$
  \mathcal E_{n-1}(f_{n-1}) = \mathcal E_n(f_n),
  $$ where  $f_n\colon V_n \to \mathbb R$
  denotes an extension which
  satisfies:
  \begin{enumerate}
  \item[(a)]
   $f_n(x) = f_{n-1}(x)$ for $x\in V_{n-1}$;  and
   \item[(b)]
  $f_n$ satisfying (a)   minimizes  $\mathcal E_n(f_n)$ (i.e., $\mathcal E_n(f_n) = \min_{f\in \ell^2(V_n)} \mathcal E_n(f)$).
   \end{enumerate}
   The following is shown  in  \cite{strichartz-book}, for example.
   \begin{lemma}
The family $(\mathcal E_n)_{n=0}^\infty$ is consistent if we choose $c_n = \left( \frac{5}{3}\right)^n$
 in \eqref{eq:energyform}.
  \end{lemma}

  The proof of this lemma is based on solving families of simultaneous equations arising from (a) and (b).
  We can now define a bilinear form for functions on $\mathcal T$ using
  the consistent family of bilinear forms $(\mathcal E_n)_{n=0}^\infty$.
  \begin{definition}
   For any continuous function $f\colon \mathcal T \to \mathbb R$ we can associate the limit
   $$\mathcal E(f) := \lim_{n \to +\infty} \mathcal E_n(f) \in [0, +\infty]$$
    and let $\hbox{\rm dom}(\mathcal E) = \{f \in C(\mathcal T) : \mathcal E(f)  < +\infty\}$.
  \end{definition}
  
  \begin{remark}
  We can  consider eigenfunctions $f \in \hbox{\rm dom}(\mathcal E)$ which satisfy Dirichlet boundary conditions (i.e., $f|V_0 = 0$).
\end{remark}

 \subsection{Laplacian for \texorpdfstring{$\mathcal T$}{the Sierpi\'nski gasket}}

To define the Laplacian $\Delta_{\mathcal T}$  the last ingredient is to  consider an inner product defined
using the  natural   measure $\mu$ on the Sierpi\'nski triangle  $\mathcal T$.

    \begin{definition}
    Let $\mu$ be the natural measure on $\mathcal T$ such that
    $$
  \mu\left(  T_{i_1} \circ \cdots \circ T_{i_n} \hbox{\rm co}(V_0) \right) = \frac{1}{3^n} \quad \hbox{ for } i_1, \dots, i_n \in \{1,2,3\},
    $$
    where $\hbox{\rm co}(V_0)$ is the convex hull of $V_0$, i.e., the filled-in triangle.
    \end{definition}

    In particular, $\mu$ is the Hausdorff measure for $\mathcal T$, and the unique measure on $\mathcal T$ for which
$$T_i^*\mu = \frac{1}{3}\mu 
\quad \hbox{ for } i=1,2,3.
$$
   The subspace $\hbox{\rm dom}(\mathcal E) \subset L^2(\mathcal T, \mu)$  is a Hilbert space.
Using the measure $\mu$ and the bilinear form $\mathcal E$ we recall  the   definition of the Laplacian $\Delta_{\mathcal T}$.

\begin{definition}
   For $u\in \hbox{\rm dom}(\mathcal E)$ which vanishes on $V_0$
   we can define the Laplacian to be a continuous function
    $\Delta_{\mathcal T} u\colon \mathcal T  \to \mathbb R$ such that
    $$\mathcal E(u,v) = -\int (\Delta_{\mathcal T} u) v d\mu$$ for any
   $v \in \hbox{\rm dom}(\mathcal E)$.
\end{definition}





\begin{remark}
For each finite graph $X_n$, the spectrum $\sigma(-\Delta_{X_n})$
for the graph Laplacian $\Delta_{X_n}$ will consist of a finite number of solutions of the eigenvalue equation
\begin{equation}\label{eq:eigvec_n}
\Delta_{X_n} f + \lambda f = 0.
\end{equation}
This is easy to see because $V_n$ is finite and thus the space $\ell^2(V_n)$ is finite-dimensional
and so the graph Laplacian can be represented as a matrix.
There is then  an alternative  pointwise formulation of the Laplacian
of the form
\begin{equation}\label{eq:laplacian_lim}
\Delta_{\mathcal T} u(x) = \frac{3}{2} \lim_{n \to +\infty}
5^n \Delta_{X_n}u(x)
\end{equation}
where $x \in \bigcup_{n=1}^\infty V_n \setminus V_0$.   The eigenvalue equation
$\Delta_{\mathcal T} u +  \lambda u = 0$ then has admissible solutions  provided  $u, \Delta_{\mathcal T} u \in C(\mathcal T)$.
A result of Kigami is that  $u \in \hbox{\rm dom}(\mathcal E)$ if and only if the convergence  in \eqref{eq:laplacian_lim} is uniform \cite{kigami-1}.
\end{remark}

 \subsection{Spectral decimation for \texorpdfstring{$\sigma(-\Delta_{\mathcal T})$}{the Sierpi\'nski gasket}}
 \label{sec:spec_dec}
We begin by briefly recalling the fundamental notion of spectral decimation introduced by \cite{rammal,rato,alexander},
 which describes the spectrum $\sigma(-\Delta_{\mathcal T})$.

\begin{figure}
\centerline{
\includegraphics[height=3.5cm, angle=-0]{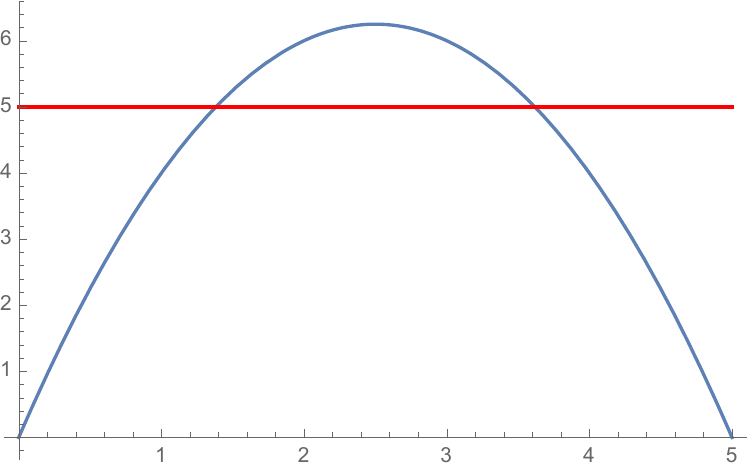}
\hskip 1cm
\includegraphics[height=3.5cm, angle=-0]{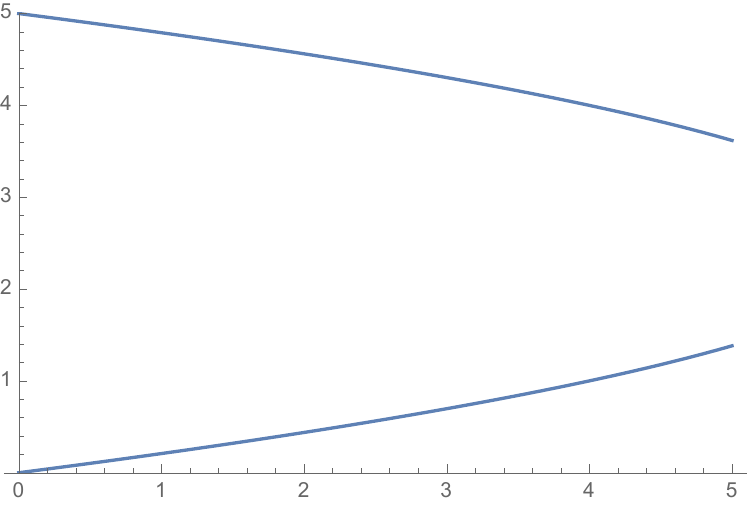}
}
\caption{The polynomial $R_{\mathcal T}(x)$ and the contracting inverse branches
$S_{-1,\mathcal T}$ and $S_{+1,\mathcal T}$ for the Sierpi\'nski triangle $\mathcal T$.}
\label{fig:rt_inv_branches}
\end{figure}

\begin{definition}
Given the polynomial $R_{\mathcal T}\colon [0,5] \to \mathbb R$
defined by
 $$R_{\mathcal T}(x) = x(5-x), $$
 we can associate local inverses (see Figure \ref{fig:rt_inv_branches})
 $S_{-1, \mathcal T}, S_{+1, \mathcal T}\colon [0,5] \to [0,5]$ of the form
\begin{equation}\label{eq:inv_branches}
S_{\epsilon, \mathcal T}(x) = \frac{5}{2} + \frac{\epsilon}{2} \sqrt{25 - 4 x}
\quad \hbox{ for } \epsilon = \pm 1.
\end{equation}
\end{definition}

The process of spectral decimation (see \cite[\S 3.2]{strichartz-book} or \cite{fs})
describes the eigenvalues of $-\Delta_{\mathcal T}$ as renormalized limits of (certain) eigenvalue sequences of $-\Delta_{X_n}, n \in \mathbb N$.
These eigenvalues, essentially, follow the recursive equality
$\lambda_{n+1} = S_{\pm1, \mathcal T}(\lambda_n)$,
while the corresponding eigenfunctions of $-\Delta_{X_{n+1}}$ are such that their
restrictions to $V_n$ are eigenfunctions for $-\Delta_{X_n}$.
Thus, the eigenvalue problem can be solved inductively, constructing solutions $f$ to
the eigenvalue equation \eqref{eq:eigvec_n} at level $n+1$ from solutions
at level $n \in \mathbb N$.
 The values of $f$ at vertices in $V_{n+1}\setminus V_n$ are
 obtained from solving the additional linear equations that arise from
 the eigenvalue equation $\Delta_{X_{n+1}}f + \lambda f = 0$, which allows for
 exactly two solutions. The exact limiting process giving rise to eigenvalues
 of $-\Delta_{\mathcal T}$ is described by the following result.

\begin{proposition}[\cite{fs,rammal,bellissard}]
Every solution $\lambda \in \mathbb{R}$ to the eigenvalue equation
\begin{equation}\label{eq:eigvec}
\Delta_{\mathcal T} u  +  \lambda u = 0
\end{equation}
can be written as
\begin{equation}\label{eq:eigvallim}
\lambda = \frac{3}{2} \lim_{m \to+\infty} 5^{m+c} \lambda_m,
\end{equation}
for a sequence $(\lambda_m)_{m \geq m_0}$ and a positive integer  $c \in \mathbb{N}_0$ satisfying
\begin{enumerate}
\item $\lambda_{m_0} = 2$ and $c = 0$, or $\lambda_{m_0} = 5$ and $c \geq 1$,
or $\lambda_{m_0} = 3$ and $c \geq 2$;
\item $\lambda_{m} = \lambda_{m+1} (5 - \lambda_{m+1}) = R_{\mathcal T}(\lambda_{m+1})$ for all $m \geq m_0$; and
\item the limit \eqref{eq:eigvallim} is finite.
\end{enumerate}
Conversely, the limit of every such sequence gives rise to a solution of \eqref{eq:eigvec}.
\end{proposition}

We remark that equivalently, the sequence $(\lambda_m)_{m \geq m_0}$
could be described
recursively as $\lambda_{m+1} = S_{\epsilon_m,\mathcal{T}}(\lambda_m)$
where $\epsilon_m \in \{\pm 1 \}$ for $m \geq m_0$.
The finiteness of the limit \eqref{eq:eigvallim} is equivalent to
there being an $m' \geq m_0$ such that $\epsilon_m = -1$ for all $m \geq m'$.

\subsection{Spectrum of the Laplacian for Sierpi\'nski lattices}
\label{sec:spec_lattice}
For a Sierpi\'nski lattice, we define the Laplacian
$\Delta_{\mathcal L}$ 
by
$$
(\Delta_{\mathcal L}f)(x) = s_x \sum_{y\sim x} (f(y)-f(x))
$$
with
$$
s_x
=
\begin{cases}
2 &\hbox{ if $x$ is a boundary point},\cr
1 &\hbox{ if $x$ is not a boundary point},\cr
\end{cases}
$$
which is a well-defined and bounded operator from $\ell^2(V^\infty)$ to itself (this follows from the fact that each vertex of $\mathcal L$ has at most $4$ neighbours).

\begin{remark} \label{rem:seq_indep}
We note that our definition of $V^\infty$ and $\mathcal L$ depended on the choice
of a sequence $(\omega_n)_{n \in \mathbb N}$, and graphs resulting from different
sequences are typically not isometric \cite[Lemma 2.3(ii)]{teplaev}.
On the other hand, the spectrum $\sigma(-\Delta_{\mathcal L})$ turns out to be
independent of this choice (see \cite[Remark 4.2]{teplaev} or \cite[Proposition 1]{sabot2}).
\end{remark}

The operator $-\Delta_{\mathcal L}\colon \ell^2(V^\infty) \to \ell^2(V^\infty)$ has a more complicated spectrum which depends on the following definition.

\begin{definition}[cf.~\cite{falconer}] We define the Julia set associated to $R_{\mathcal T}$ to be the
smallest non-empty closed set $\mathcal J_{\mathcal T} \subset [0, 5]$ such that
$$\mathcal J_{\mathcal T} = S_{-1, \mathcal T}(\mathcal J_{\mathcal T}) \cup S_{+1, \mathcal T} (\mathcal J_{\mathcal T}).$$
\end{definition}

This leads to the following description of the spectrum  $\sigma(-\Delta_{\mathcal L})$.

\begin{proposition}[{\cite[Theorem 2]{teplaev}}]
The operator $-\Delta_{\mathcal L}$ on $\ell^2(V^\infty)$ is bounded, non-negative and self-adjoint and has
spectrum  $$\sigma(-\Delta_{\mathcal L}) = \mathcal J_{\mathcal T} \cup \left(\{6\} \cup  \bigcup_{n=0}^{\infty} R^{-n}(\{3\})\right).$$
\end{proposition}

This immediately leads to the following.

\begin{corollary}\label{cor:dim_eq_lattice}
We have that
 $\dim_H(\sigma(-\Delta_{\mathcal L})) = \dim_H(\mathcal J_{\mathcal T})$.
\end{corollary}

Thus estimating the Hausdorff dimension of the spectrum
$\sigma(-\Delta_{\mathcal L})$
is equivalent to estimating that of the Julia set $\mathcal J_{\mathcal T}$. The following provides a related application.

\begin{figure}
\centerline{\includegraphics[height=3.5cm, angle=-0]{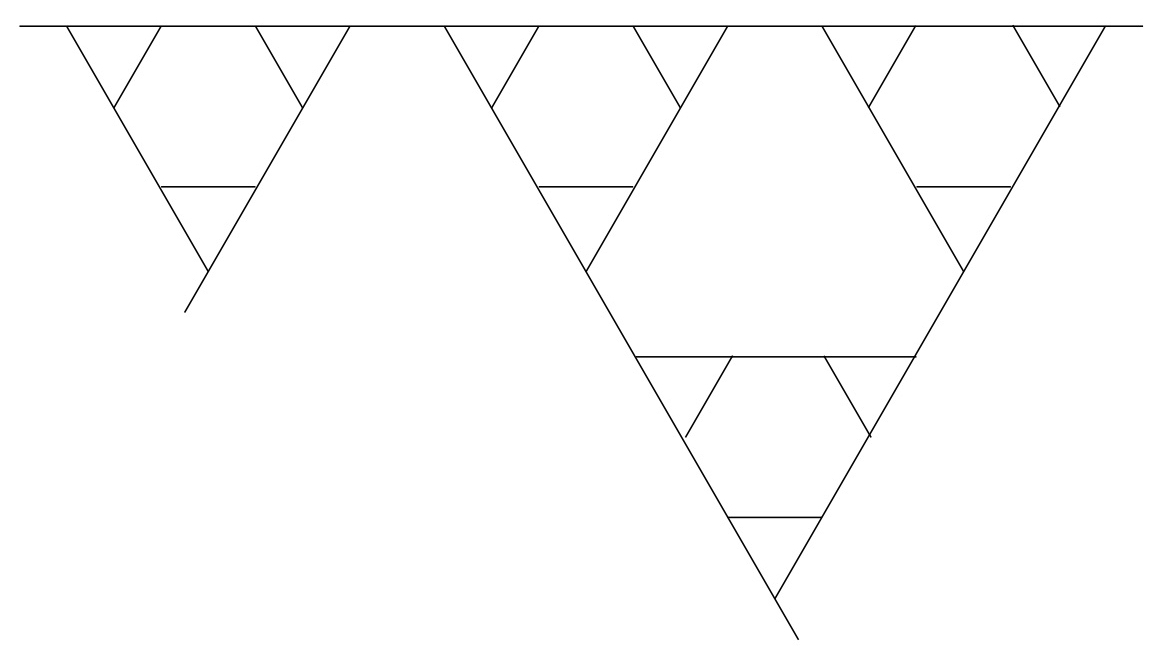}}
\caption{The Pascal graph.}
\label{fig:pascal}
\end{figure}

\begin{example}[Pascal graph]
Consider the Pascal graph $\mathcal P$ \cite{quint}, which is an infinite $3$-regular
graph, see Figure \ref{fig:pascal}.
Its edges graph is the Sierpi\'nski lattice $\mathcal L$, and
as was shown by Quint \cite{quint}, the spectrum 
$\sigma (-\Delta_{\mathcal P})$
of its Laplacian $-\Delta_{\mathcal P}$ is
the union of a countable set and the Julia set of a certain polynomial (affinely) conjugated to $R_{\mathcal T}$.
From this we deduce that
$$\dim_H(\sigma (-\Delta_{\mathcal P}))
 =  \dim_H({\mathcal J}_{\mathcal P}) =  \dim_H({\mathcal J}_{\mathcal L})
=\dim_H(\sigma (-\Delta_{\mathcal L})),$$
which we estimate in Theorem \ref{thm:main}.
\end{example}



\subsection{Spectrum of the Laplacian for infinite Sierpi\'nski gaskets}
\label{sec:spec_inf_gasket}

We finally turn to the case of an infinite Sierpi\'nski gasket $\mathcal T^\infty$.
The Laplacian $\Delta_{\mathcal T^\infty}$ is an operator with a domain in
$L^2(\mathcal T^\infty, \mu^\infty)$. Here $\mu^\infty$ is the
natural measure on $\mathcal T^\infty$, whose restriction to $\mathcal T$ equals $\mu$, and such that any two isometric sets are of equal measure (see \cite{teplaev}).

Remark \ref{rem:seq_indep} applies almost identically also to the Sierpi\'nski gasket case:
$\mathcal T^\infty$ depends non-trivially on the choice of a sequence $\omega$ in its definition,
and different sequences typically give rise to non-isometric gaskets,
with the boundary of $\mathcal T^\infty$ empty if and only if $\omega$ is eventually constant \cite[Lemma 5.1]{teplaev}.
The spectrum $\sigma(-\Delta_{\mathcal T^\infty})$, however,
is independent of $\omega$ (even if the spectral decomposition is not, see \cite[Remarks 5.4]{teplaev} or \cite[Proposition 1]{sabot2}).
Using the notation
$$
\mathcal R (z) = \lim_{n\to \infty} 5^n (S_{-1,\mathcal T})^n(z),
$$
we have the following result on the spectrum $\sigma(-\Delta_{\mathcal T^\infty})$.

\begin{proposition}[{\cite[Theorem 4]{teplaev}}]\label{prop:union_inf_gasket}
The operator $-\Delta_{\mathcal T^\infty}$  is an unbounded self-adjoint operator from a dense domain in
$L^2(\mathcal T^\infty, \mu^\infty)$ to $L^2(\mathcal T^\infty, \mu^\infty)$.
Its spectrum is $\sigma(-\Delta_{\mathcal T^\infty})= \mathcal J^\infty \cup \Sigma_3^\infty$
with
$$ \mathcal J^\infty = \bigcup_{n=-\infty}^\infty 5^n \mathcal R(\mathcal J_{\mathcal T})
\quad \hbox{ and } \quad
\Sigma_3^\infty = \bigcup_{n=-\infty}^\infty 5^n \mathcal R(\Sigma_3),$$
where $\Sigma_3 = \bigcup_{n=0}^{\infty} R^{-n}(\{3\})$.
\end{proposition}

A number of generalizations of this result for other unbounded nested fractals
have been proved, see e.g. \cite{sabot,st}. The proposition immediately yields the following corollary.

\begin{corollary}\label{cor:eq_dim_inf_gasket}
We have that
 $\dim_H(\sigma(-\Delta_{\mathcal T^\infty})) = \dim_H(\mathcal J_{\mathcal T})$.
\end{corollary}

Thus estimating the Hausdorff dimension of the spectrum
$\sigma(-\Delta_{\mathcal T^\infty})$ is again equivalent
to estimating the Hausdorff dimension of the Julia set $\mathcal J_{\mathcal T}$.


\section{Related results for other gaskets and lattices} \label{sec:related_results}
In this section we describe other examples of fractal sets to which the same approach can be applied.
In practice the computations may be more complicated, but the same basic method still applies.

\subsection{Higher-dimensional infinite Sierpi\'nski gaskets}
Let
$d \geq 2$ and
$T_i\colon \mathbb R^d \to \mathbb R^d$ be contractions defined by
$$
T_i(x_1, \ldots, x_d) = \left(
\frac{x_1}{2}, \ldots, \frac{x_d}{2}
\right) +
\frac{1}{2} e_i, \quad \hbox{ for } i = 1, \ldots, d,
$$
where $e_i$ is the $i$th coordinate vector.
The $d$-dimensional Sierpi\'nski gasket $\mathcal T^d \subset \mathbb R^d$
is the  smallest non-empty closed set such that
$$
\mathcal T^d = \bigcup_{i=1}^d T_i (\mathcal T^d).
$$

In \cite{rammal}, the analogous results are presented for the spectrum of the Laplacian
$\Delta_{\mathcal T^d}$ associated to the corresponding
Sierpi\'nski gasket $\mathcal T^d \subset \mathbb R^d$
in $d$ dimensions $(d \geq 3$).

\begin{definition}
For a sequence $(\omega_n)_{n \in \mathbb N} \subset \{1, \ldots, d\}^{\mathbb N}$ we can define an \textit{infinite Sierpi\'nski gasket in $d$ dimensions} as
$$
\mathcal T^{d, \infty} = \bigcup_{n=1}^\infty T_{\omega_1}^{-1}\circ \cdots \circ T_{\omega_n}^{-1} (\mathcal T^d).
$$
\end{definition}

As before, we can associate a Julia set ${\mathcal J}_{\mathcal T^d}$ and consider its Hausdorff dimension
$ \dim_H({\mathcal J}_{\mathcal T^d})$.
More precisely,
  in each  case, we can consider the decimation polynomial
$R_{\mathcal T^d}\colon [0,3+d] \to \mathbb R$ defined by
$$R_{\mathcal T^d}(x)= x ((3+d)-x),$$
with two local inverses $S_{\pm1, \mathcal T^d}\colon [0, 3+d] \to [0,3+d]$ given by
$$
 S_{\epsilon,\mathcal T^d} (x) = \frac{1}{2}  \left(
3+d + \epsilon \sqrt{9+6d +d^2-4x}
\right) \quad \hbox{ with } \epsilon = \pm 1.
$$
Let $ {\mathcal J}_{\mathcal T^d} \subset[0,3+d]$ be
the limit set of these two contractions, i.e.,
the smallest non-empty closed set such that 
$${\mathcal J_{\mathcal T^d}} = S_{-1, \mathcal T^d}({\mathcal J}_{\mathcal T^d}) \cup S_{+1, \mathcal T^d} ({\mathcal J}_{\mathcal T^d}).$$

\begin{theorem} \label{thm:td_dim}
The Hausdorff dimension $\dim_H(\mathcal J_{\mathcal T^d})$ of the Julia set $\mathcal J_{\mathcal T^d}$ for
$d \in \{2, \ldots, 10\}$ associated to the Sierpi\'nski gasket in $d$ dimensions
is given by the values in Table \ref{tab:dim_jtd}, accurate to the number of decimals stated.
\begin{table}[h!]
\begin{center}
\begin{tabular}{ |c|c| }
 \hline
$d$ &$ \dim_H({\mathcal J}_{\mathcal T^d})$  \\
 \hline
2 &  0.55161856837246 \ldots \\
3 &  0.45183750018171 \ldots \\
4 &  0.39795943979056 \ldots \\
5 &  0.36287714809375 \ldots \\
6 &  0.33770271892130 \ldots \\
7 &  0.31850809575800 \ldots \\
8 &  0.30324865557723 \ldots \\
9 &  0.29074069840192 \ldots \\
10 & 0.28024518050407 \ldots \\
 \hline
\end{tabular}
\end{center}
\caption{The Hausdorff dimension of ${\mathcal J}_{\mathcal T^d}$ for $2 \leq d \leq 10$.}
\label{tab:dim_jtd}
\end{table}
\end{theorem}

The proof uses the same algorithmic method as that of Theorem \ref{thm:main}, see Section \ref{sec:dim_estimate}.

\begin{remark}
By arguments developed in \cite{fs} and \cite{sabot2}, one can deduce that similarly
to Proposition \ref{prop:union_inf_gasket} and Corollary \ref{cor:eq_dim_inf_gasket}, the Hausdorff dimensions of the spectrum of the appropriately defined
Laplacian on $\mathcal T^{d,\infty}$ and the Julia set $ \dim_H({\mathcal J}_{\mathcal T^d})$ coincide.
\end{remark}

We can observe empirically from the table that the dimension decreases as $d \to +\infty$.
The following simple lemma confirms that $\lim_{ d \to +\infty}  \dim_H(\mathcal J_{\mathcal T^{d}}) = 0$ with explicit bounds.

\begin{lemma} As $d \to +\infty$ we can bound
$$
\frac{\log 2}{\log (d+3)}
\leq \dim_H(\mathcal J_{\mathcal T^{d}})
\leq \frac{2 \log 2}{\log (d+3) + \log (d-1)}
.
$$
\end{lemma}

\begin{proof}
 We can write
 $$I_1:= R_{\mathcal T^d}^{-1}([0,3+d]) \cap
 \left[ 0, \frac{3+d}{2} \right] = \left[0, \frac{3+d}{2}
\left(1-  \sqrt{1- \frac{4}{3+d} } \right)
 \right].
 $$
 Thus for $x \in I_1$ we have bounds
 $$\sqrt{(3+d)(d-1)}\leq |R_{\mathcal T^d}'(x)| \leq 3+d. $$
 Similarly, we can define $I_2:= R_{\mathcal T^d}^{-1}([0,3+d]) \cap
 \left[\frac{3+d}{2}, 3+d \right]$ and obtain the same bounds on $|R_{\mathcal T^d}'(x)|$ for $x\in I_2$.
 In particular, we can then bound
\[
 \frac{\log 2}{\log (3+d)} \leq
 \dim_H(\mathcal J_{\mathcal T^{d}}) \leq
\frac{2 \log 2}{\log (3+d)+ \log (d-1) }.\qedhere
\]
\end{proof}


\subsection{Post-critically finite self-similar sets}
The method of spectral decimation used for the Sierpi\'nski gasket by
Fukushima and Shima \cite{fs}, was extended by Shima \cite{shima} to
post-critically finite self-similar sets and thus provided a method for analyzing the spectra of their Laplacians.

\begin{definition}
Let $\Sigma = \{1, \ldots, k\}^{\mathbb Z_+}$ be the space of (one-sided) infinite sequences  with the Tychonoff product topology, and $\sigma$ the usual left-shift map on $\Sigma$.

Let $T_1, \ldots, T_k\colon \mathbb R^d \to \mathbb R^d$ be contracting similarities and
let $\mathcal X$ be the limit set, i.e., the smallest closed subset with $\mathcal X = \bigcup_{i=1}^k T_i(\mathcal X)$.
Let $\pi\colon \Sigma \to \mathcal X$ be the natural continuous map defined by
$$
\pi\left((w_n)_{n=0}^\infty\right) = \lim_{n \to +\infty} T_{w_0} T_{w_1} \cdots T_{w_n}(0).
$$
We say that $\mathcal X$ is post-critically finite if
$$
\#
\left(
\bigcup_{n=0}^\infty \sigma^n
\left\{
(w_n) \in \Sigma : \pi(w_n) \in K
\right\}
\right) < +\infty
$$
where $K = \bigcup_{i\neq j} T_i \mathcal X \cap T_j \mathcal X$.

\end{definition}

\begin{figure}
\centerline{
\begin{tikzpicture}[scale=0.35]
\draw[thick, black] (0,0) -- (10,0)--(5,8.66)--(0,0);
\node at (5,-1.5){$X_0$};
\end{tikzpicture}
\hskip 1cm
\begin{tikzpicture}[scale=0.35]
\draw[thick, black] (0,0) -- (10,0)--(5,8.66)--(0,0);
\draw[thick, black] (3.333,0) -- (5, 2.88675)--(1.6666, 2.88675)--(3.333,0);
\draw[thick, black] (6.666,0) -- (8.333, 2.88675)--(5.0, 2.88675)--(6.666,0);
\draw[thick, black] (3.333,5.7735) -- (6.666, 5.7735)--(5, 2.88675)--(3.333,5.7735);
\node at (5,-1.5){$X_1$};
\end{tikzpicture}
\includegraphics[height=3.8cm, angle=0, trim={-5cm -4.4cm 0 1cm}]{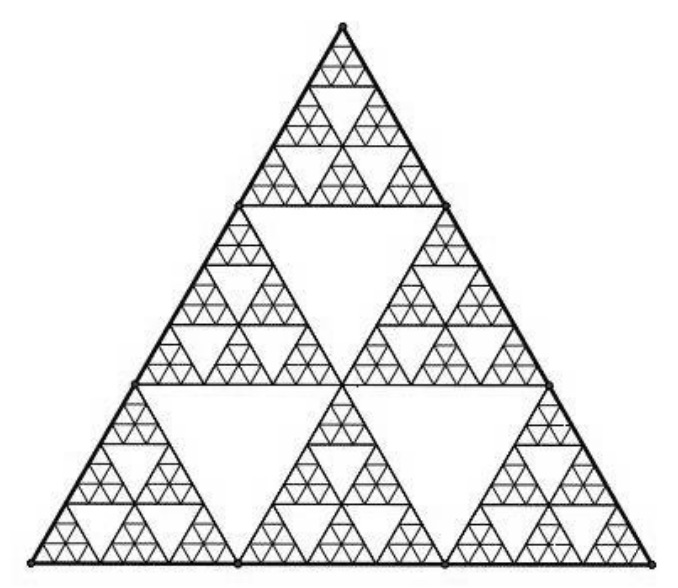}
}
\caption{The first two graphs for $SG_3$ (left, centre) and the $SG_3$ gasket (right).}
\label{fig:sg3}
\end{figure}
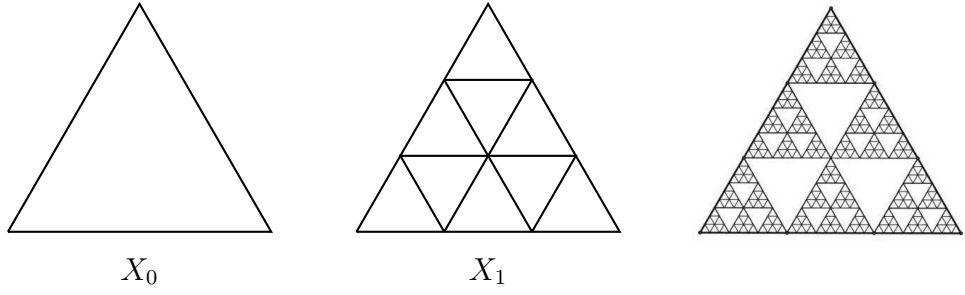

\begin{figure}
\centerline{
\includegraphics[height=3.5cm, angle=-0]{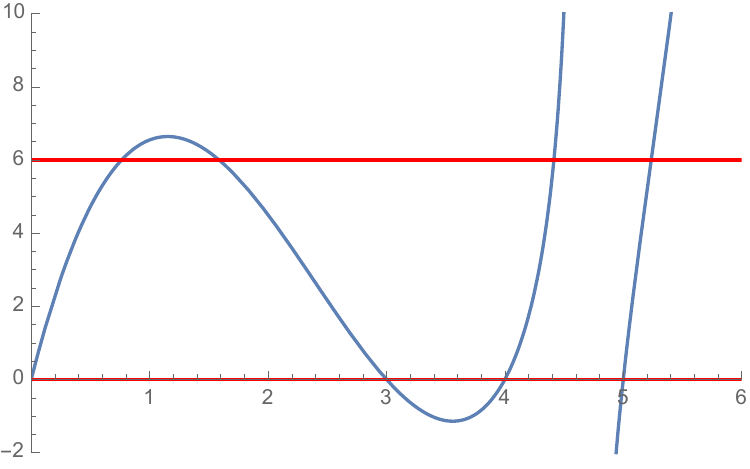}
\hskip 1cm
\includegraphics[height=3.5cm, angle=-0]{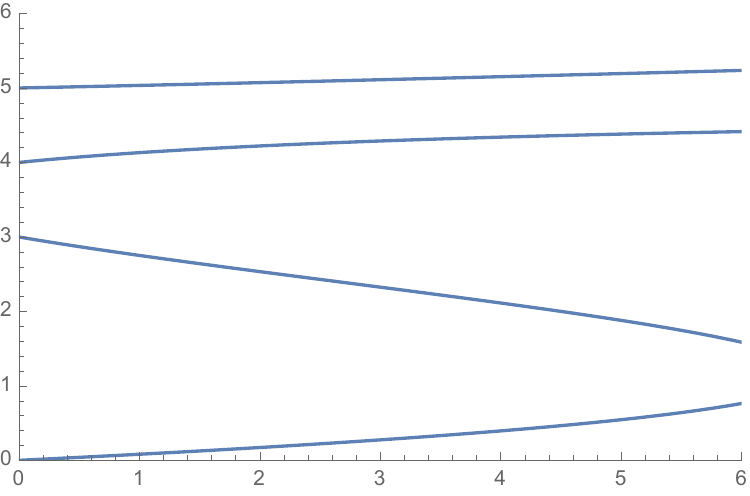}
}
\caption{The function $R_{SG_3}(x)$ and the four contracting inverse
branches for the $SG_3$ gasket.}
\label{fig:sg3_inv_branches}
\end{figure}

The original Sierpi\'nski triangle $\mathcal T$ is
an example of a limit set which is
post-critically finite.  So is the following variant on the Sierpi\'nski triangle.

\begin{example}[$SG_3$ gasket] \label{ex:sg3}
We can consider the Sierpi\'nski gasket $SG_3$ (see Figure \ref{fig:sg3}) which is the smallest
non-empty closed set $\mathcal X_{SG_3}$ such that
$
\mathcal X_{SG_3} = \bigcup_{i=1}^6 T_i \mathcal X_{SG_3}
$
where
$$
T_j (x,y) = p_j + \left( \frac{x}{3}, \frac{y}{3} \right) \quad \hbox{ for } j =1, \ldots, 6,
$$
with
\begin{align*}
p_1 = (0,0),
p_2 = \left(\frac{1}{3},0 \right),
p_3 = \left(\frac{2}{3},0 \right),
p_4 = \left(\frac{1}{6},\frac{1}{2\sqrt{3}} \right),
p_5 = \left(\frac{1}{2},\frac{1}{2\sqrt{3}} \right),
p_6 = \left(\frac{1}{3},\frac{1}{\sqrt{3}} \right).
\end{align*}

\noindent In this case we can  associate the decimation rational function $R_{SG_3}\colon [0,6] \to [0,6]$ given by
$$
R_{SG_3}(x) = \frac{3x (5-x) (4-x) (3-x)}{14-2x},
$$
for which there are four local inverses $S_{j, SG_3}$ (for $j=1,2,3,4$) \cite{ds},
see Figure \ref{fig:sg3_inv_branches}.
The associated Julia set ${\mathcal J}_{SG_3}$,
which is the smallest non-empty closed set such that
${\mathcal J}_{SG_3} = \bigcup_{j=1}^4 S_{j, SG_3} ({\mathcal J}_{SG_3})$,
has Hausdorff dimension $ \dim_H({\mathcal J}_{SG_3})$.

Using Mathematica with a sufficiently high precision setting (see Example
\ref{ex:sg3_numerics} for more details) we can numerically compute the Hausdorff dimension of the Julia set ${\mathcal J}_{SG_3}$ associated to
the Sierpi\'nski gasket $SG_3$ to be
$$
\dim_H({\mathcal J}_{SG_3}) =
0.617506301862352229042494874316407096341976 \ldots
$$
\end{example}


\begin{example}[Vicsek graph]
The Vicsek set $X_{\mathcal V}$ is the smallest non-empty closed set such that 
$
X_{\mathcal V} = \bigcup_{j=1}^5 T_j(X_{\mathcal V})
$
where 
$$
T_j(x,y) = p_j + \left(\frac{x}{3}, \frac{y}{3}\right) \quad
\hbox{ for } j=1, \ldots, 5,
$$
with
\[
p_1 = (0,0),~
p_2 = \left(\frac{2}{3},0 \right),~
p_3 = \left(\frac{2}{3},\frac{2}{3} \right),~
p_4 = \left(0,\frac{2}{3} \right),~
p_5 = \left(\frac{1}{3},\frac{1}{3} \right).
\]

In this case, studied in \cite[Example 6.3]{mt2}, 
one has that $R_{\mathcal V}\colon[-1,0 ] \to \mathbb R$ is given by
$$
R_{\mathcal V}(z) = z(6z +3) (6z+5),
$$
with three inverse branches $S_1, S_2, S_3\colon [-1,0] \to [-1,0]$ given by
$$
\begin{aligned}
S_1(x)&= \frac{1}{36} \left(i(\sqrt{3} + i) t(x) - \frac{19(1 + i \sqrt{3})}{t(x)} - 16\right),\cr
S_2(x)&= \frac{1}{36} \left(-i(\sqrt{3} - i) t(x) - \frac{19(1 - i \sqrt{3})}{t(x)} - 16\right),\cr
S_3(x)&= \frac{1}{18} \left(t(x) + \frac{19}{t(x)} - 8\right),
\end{aligned}
$$
where $t(x) = \left(9\cdot (81x^2 + 56x - 75)^{1/2} + 81x+28\right)^{1/3}$.
The associated Julia set ${\mathcal J}_{\mathcal V}$ is the smallest
non-empty closed set such that
${\mathcal J}_{\mathcal V} = \bigcup_{j=1}^3 S_{j, \mathcal V} ({\mathcal J}_{\mathcal V})$. The following theorem is proved similarly to Theorem \ref{thm:main},
as described in Section \ref{sec:dim_estimate}.

\begin{theorem} \label{thm:vicsek_dim}
The Hausdorff dimension of the Julia set ${\mathcal J}_{\mathcal V}$ is
$$
\dim_H({\mathcal J}_{\mathcal V}) = 0.49195457005266 \ldots,
$$
accurate to the number of decimals stated.
\end{theorem}

\end{example}

\begin{remark}
Analogously to the case of the Sierpi\'nski lattice $\mathcal L$, we
can define lattices $\mathcal L_{SG_3}$ and $\mathcal L_{\mathcal V}$
for the $SG_3$ and Vicsek sets from the previous two examples, as well
as corresponding graph Laplacians $\Delta_{\mathcal L_{SG_3}}$ and
$\Delta_{\mathcal L_{\mathcal V}}$. The Hausdorff dimensions of their spectra
can again be directly related to those of the respective
Julia sets $\mathcal J_{SG_3}$ and $\mathcal J_{\mathcal V}$.
By \cite[Theorem 5.8]{mt2},
one has that
${\mathcal J}_{SG_3} \subseteq \sigma(-\Delta_{\mathcal L_{SG_3}}) \subseteq {\mathcal J}_{SG_3} \cup \mathcal{D}_{SG_3}$
and ${\mathcal J}_{\mathcal V} \subseteq \sigma(-\Delta_{\mathcal L_{\mathcal V}}) \subseteq {\mathcal J}_{\mathcal V} \cup \mathcal{D}_{\mathcal V}$,
where $\mathcal{D}_{SG_3}$ and $\mathcal{D}_{\mathcal V}$ are countable sets.
It follows, analogously to Corollary \ref{cor:dim_eq_lattice}, that
$\dim_H(\sigma(-\Delta_{\mathcal L_{SG_3}})) = \dim_H({\mathcal J}_{SG_3})$ and
$\dim_H(\sigma(-\Delta_{\mathcal L_{\mathcal V}})) = \dim_H({\mathcal J}_{\mathcal V})$.
\end{remark}

\begin{remark}
Other examples to which the same method could be applied include 
 the modified Koch curve (see  \cite{mt},  \cite{malozemov})
 for which the associated rational function is
$
R(x) =9 x (x-1) (x - \frac{4}{3})(x - \frac{5}{3})/(x - \frac{3}{2}).
$
More  families of such examples can be found in \cite{t-can}.
\end{remark}

\begin{remark}
The spectral decimation method can also apply to some non-post-critically finite examples, such as the
 diamond fractal \cite{ki}, for which the associated polynomial is $R(x) = 2 x (2 + x)$.
  On the other hand, there  are symmetric fractal sets which do not admit spectral decimation, such as the pentagasket, as studied in \cite{adams}.
\end{remark}





\section{Dimension estimate algorithm for Theorem \ref{thm:main}}
\label{sec:dim_estimate}

This section is dedicated to finishing the proof of Theorem \ref{thm:main},
by describing an algorithm yielding estimates (with rigorous error bounds)
for the values of the Hausdorff dimension.

By the above discussion we have reduced the estimation of the
Hausdorff dimensions of $\sigma(-\Delta_{\mathcal L})$ and
$\sigma(-\Delta_{\mathcal T^\infty})$ to that of
$\dim_H({\mathcal J}_{\mathcal T})$ for the limit set ${\mathcal J}_{\mathcal T}$
associated to $S_{\pm1,\mathcal T}$ from \eqref{eq:inv_branches} (and similarly for the other examples).  Unfortunately, since the maps
$S_{\pm1, \mathcal T}$ are non-linear it is not possible to give an explicit closed form for the value $\dim_H(\sigma(-\Delta_{\mathcal T}))= \dim_H({\mathcal J}_{\mathcal T})$.
Recently developed simple methods make the numerical estimation of this value relatively easy to implement, which we summarize in the following subsections.

\subsection{A functional characterization of dimension}
Let $\mathcal B = C(I)$ be the Banach space of continuous  functions on the interval
$I = [0,5]$ with the norm $\|f\|_\infty  = \sup_{x \in I} |f(x)|$.

\begin{definition}\label{defn:op}
Let $\mathcal L_t$ (for $t \geq 0$) be the
{\it transfer operator}
defined by
$$
\mathcal L_t f(x) = |S_{-1,\mathcal T}'(x)|^t f(S_{-1,\mathcal T}(x)) + |S_{+1,\mathcal T}'(x)|^t f(S_{+1,\mathcal T}(x))
$$ where $f \in \mathcal B \hbox{ and }x\in I$, and $S_{\pm,\mathcal T}$ are as in \eqref{eq:inv_branches}.
\end{definition}

It is well known that the transfer operator $\mathcal L_t$ (for $t \geq 0$) is a well-defined positive bounded operator from $\mathcal B$ to itself. 
To make use of the results in the previous sections, we employ the following   ``min-max method'' result:

\begin{lemma}[\cite{pv}]\label{minmax}
Given  choices of $0 < t_0 < t_1 < 1$ and strictly positive continuous   functions   $f, g\colon I \to \mathbb R^+$  with
  \begin{equation}\label{eq:infsupcond}
  \inf_{x\in I}\frac{\mathcal L_{t_0} f(x)}{f(x)} > 1 \quad \hbox{ and } \quad \sup_{x\in I}\frac{\mathcal L_{t_1} g(x)}{g(x)} < 1,
  \end{equation}
 then $t_0 < \dim_H({\mathcal J}_{\mathcal T}) < t_1$.
 \end{lemma}

 \begin{proof} We briefly recall the proof. We require the following standard properties.
 \begin{enumerate}
 \item
For any $t\geq 0$ the operator 
 $\mathcal L_t$ has a maximal   positive simple eigenvalue $e^{P(t)}$ (with positive eigenfunction), where $P$ is the pressure function \cite{ruelle-book,pp}.
\item
$P\colon \mathbb R^+ \to \mathbb R$ is real analytic and convex \cite{ruelle-book}.
\item
The value
$t = \dim(\mathcal J_{\mathcal T})$  is the unique solution to
$P(t) = 0$, see \cite{bowen,ruelle-etds}.
\end{enumerate}
By property 1. and the first inequality in \eqref{eq:infsupcond}
we can deduce that 
\begin{equation}\label{eq:pt0}
P(t_0) = \lim_{n \to +\infty} \frac{1}{n}\log \|\mathcal L_{t_0}^nf\|_\infty > 0.
\end{equation}
By property 1. and the second inequality in \eqref{eq:infsupcond}
we can deduce that 
\begin{equation}\label{eq:pt1}
P(t_1) = \lim_{n \to +\infty} \frac{1}{n}\log \|\mathcal L_{t_1}^ng\|_\infty < 0.
\end{equation}
Comparing properties 2. and 3. with \eqref{eq:pt0} and \eqref{eq:pt1}, the result follows.
 \end{proof}

 \subsection{Rigorous verification of minmax inequalities}
Next, we explain how we rigorously verify the conditions of Lemma~\ref{minmax} for
a function $f\colon I \to \mathbb{R}^+$, that is,
\begin{enumerate}
\item $f>0$,
\item $\inf_{x\in I} h(x)> 1$ or $\sup_{x\in I} h(x) < 1$
for $h(x) := (\mathcal{L}_t f)(x) / f(x)$.
\end{enumerate}
In order to obtain rigorous results we make use of the arbitrary
precision ball arithmetic library Arb \cite{arb},
which for a given interval $[c-r, c+r]$ and function $f$ outputs
an interval $[c'-r', c'+r']$ such that $f([c-r, c+r]) \subseteq [c'+r', c'-r']$
is guaranteed. Clearly, the smaller the size of the input interval,
the tighter the bounds on its image.
Thus, in order to verify the above conditions we partition the interval
$I$ adaptively using a bisection method up to depth $k \in \mathbb{N}_0$
into at most $2^{k}$ subintervals, and check these conditions on each subinterval.
While the first condition is often immediately satisfied
for chosen test functions $f$ on the whole interval $I$,
the second condition is much harder to check as $h$ is very close to $1$
and would require very large depth $k$.

To counteract the exponential growth of the number of required subintervals,
we use tighter bounds on the image of $h$. Clearly for $x \in [c-r, c+r]$ with $c \in \mathbb R$ and $r > 0$,
we have that $|h(x) - h(c)| \leq \sup_{y\in [c-r,c+r]} | h'(y)| r$ by the mean value theorem.
More generally, we obtain for $p \in \mathbb N$ that
$$|h(x) - h(c)| \leq
\sum_{i=1}^{p-1} | h^{(i)}(c)| r^i +
\sup_{y\in [c-r,c+r]} | h^{(p)}(y)| r^p.$$
This allows to achieve substantially tighter bounds on $h([c-r, c+r])$ while
using a moderate number of subintervals, at the cost of additionally computing the first $p$ derivatives of $h$.

\subsection{Choice of \texorpdfstring{$f$}{f} and \texorpdfstring{$g$}{g} via an interpolation method}

Here we explain how to choose suitable functions $f$ and $g$ for use
in Lemma \ref{minmax}, so that given candidate values $t_0 < t_1$   we
can confirm that $t_0  < \dim_H({\mathcal J}_{\mathcal T}) < t_1$.
Clearly, if $f$ and $g$ are eigenfunctions of $\mathcal{L}_{t_0}$ and $\mathcal{L}_{t_1}$
for the eigenvalues $\lambda_{t_0}$ and $\lambda_{t_1}$, respectively, then condition \eqref{minmax}
is easy to check. As these eigenfunctions are not known explicitly, we will use the
Lagrange-Chebyshev interpolation method to approximate the respective transfer
operators by finite-rank operators of rank $m$, and thus obtain approximations $f^{(m)}$ and $g^{(m)}$
of $f$ and $g$. As the maps $S_{\pm 1, \mathcal{T}}$ involved in the definition of the transfer operator
(Definition \ref{defn:op}) extend to holomorphic functions on suitable ellipses,
Theorem 3.3 and Corollary 3 of \cite{BJ} guarantee that
the (generalized) eigenfunctions of the finite-rank operator converge (in supremum norm) exponentialy fast in $m$ to those of the transfer operator. In particular, for large enough $m$ the functions
$f^{(m)}$ and $g^{(m)}$
are positive on the interval $I$ and are good candidates for Lemma~\ref{minmax}.

\smallskip
\noindent 
{\bf Initial choice of $m$.} We first make an initial choice of $m \geq 1$.
Let $\ell_n\colon I \to \mathbb R$ (for $n=0, \ldots, m-1$) denote the Lagrange polynomials scaled to  $[0,5]$ and let
$x_n \in [0,5]$  (for $n=0, \ldots, m-1$)  denote  the associated Chebyshev points.

\smallskip
\noindent 
{\bf Initial construction of test functions.} 
Let  $v^t = (v_i^{t})_{i=0}^{m-1}$ be the left eigenvector for the maximal eigenvalue of the $m \times m$
matrix\footnote{A fast practical implementation of this requires a
slight variation \cite[Algorithm 1]{BJ}, which can be implemented
using a discrete cosine transform.}  $M_{t}(i,j) =  (\mathcal L_t \ell_i) (x_j)$ (for $0 \leq i,j \leq m-1$) and set
$$
f^{(m)} := \sum_{i=0}^{m-1} v_i^{t_0} \ell_i \quad
\hbox{ and } \quad
g^{(m)} := \sum_{i=0}^{m-1} v_i^{t_1} \ell_i.
$$
If the choices $f = f^{(m)}$ and  $g = g^{(m)}$ satisfy the hypotheses of Lemma~\ref{minmax} (which can be checked rigorously with the method in the previous section) then
we proceed to the next step. If they do not,
 we increase $m$ and try again.

\smallskip
\noindent 
{\bf Conclusion.}   
When the hypothesis of Lemma \ref{minmax} holds then its assertion confirms that $t_0  < \dim_H(\mathcal J_{\mathcal T}) < t_1$.

\smallskip
 It remains to iteratively make the best possible choices of $t_0 < t_1$ using the following approach.

 \subsection{The bisection method}
 Fix $\epsilon > 0$. We can combine the above method of choosing $f$ and $g$ with a bisection method to  improve given lower and upper bounds $t_0$ and $t_1$
 until the latter are  $\epsilon$-close:

\smallskip
\noindent 
{\bf Initial choice.} First we can set $t_0^{(1)} = 0$ and $t_1^{(1)}=1$,  for which 
$t_0^{(1)}  < \dim_H(\mathcal J_{\mathcal T}) < t_1^{(1)}$ is trivially true.

\smallskip
\noindent 
{\bf Iterative step.}
Given 
$n \geq 0$ we assume that we have chosen
$t_0^{(n)}< t_1^{(n)}$.  We can then set $T=(t_0^{(n)}+t_1^{(n)})/2$ and
proceed as follows.
\begin{enumerate}
\item[(i)]
 If 
 $t_0^{(n)}  < \dim_H(\mathcal J_{\mathcal T}) <T$
  then set $t_0^{(n+1)} =t_0^{(n)}$ and $t_1^{(n+1)} =T$.
 \item[(ii)] If
$T  < \dim_H(\mathcal J_{\mathcal T}) <t_1^{(n)}$ then set $t_0^{(n+1)} =T$ and $t_1^{(n+1)} =t_1^{(n)}$.
 \item[(iii)]
  If $ \dim_H(\mathcal J_{\mathcal T}) = T$ then we have the final value.\footnote{In  practical implementation, the case (iii) is of no relevance, and the only
  meaningful termination condition is given by $t_1 - t_0 < \epsilon$.}
\end{enumerate}

\smallskip
\noindent 
{\bf Final choice.} Once we arrive at 
 $t_1^{(n)} - t_0^{(n)} < \epsilon$ then we can  set $t_0 = t_0^{(n)}$ and  $t_1 = t_1^{(n)}$ as the resulting upper and lower bounds for the true value of
 $\dim_H(\mathcal J_{\mathcal T})$.

\bigskip
Applying this algorithm yields the proof of Theorem \ref{thm:main} (and with the
obvious modifications also those of Theorems \ref{thm:td_dim} and \ref{thm:vicsek_dim}). Specifically,
we computed the value of $\dim_H(\mathcal J_{\mathcal T})$
efficiently to the $14$ decimal places as stated with the above method,
by setting $\epsilon = 10^{-15}$, using finite-rank approximation up to rank $m=30$,
running interval bisections for rigorous minmax inequality verification
up to depth $k=18$, i.e. using up to $2^{18}$ subintervals, and using $p=2$ derivatives.
There are of course many ways to improve accuracy further, e.g., with more computation or the use of higher derivatives.

\begin{example}[Sierpi\'nski triangle]
To cheaply obtain a more accurate estimate (albeit without the rigorous guarantee resulting from the use of ball arithmetic), we use the \textsc{MaxValue} routine from Mathematica. To get an estimate on $\dim_H({\mathcal J}_{\mathcal T})$ to $60$ decimal places, we work with $100$ decimal places as Mathematica's precision setting. Taking $m=60$ we use the bisection  method and starting from
an interval $[0.2,0.8]$ after $199$ iterations we have upper and lower bounds $t_0 \leq \dim_H({\mathcal J}_{\mathcal T}) \leq t_1$, where
$$
\begin{aligned}
t_0 =&0.5516185683724609316975708723135206545360797417440422\cr
    &\qquad 0826629\color{gray}{66000504800341581203344828264869391054705} \cr
 \end{aligned}
$$
and 
$$
\begin{aligned}
t_1 =&0.5516185683724609316975708723135206545360797417440422\cr
&\qquad 0826629\color{gray}{80935741467208321300490581993941689232122}. \cr
\end{aligned}
$$

With a little more computational effort ($200$ decimals of precision, $m = 100$, $329$ iterations) we can improve the estimate to $100$ decimal places:
$$
\begin{aligned}
t_0 =&0.55161856837246093169757087608456543417211766450713\cr
       &\quad 8868116831699168666814224190486583439508658139692\color{gray}{4}\cr
       &\quad \color{gray}{80473399364569014861603996382396316337795734913712}\cr
       &\quad \color{gray}{92389795501216939500532891268573684698907908711334}\cr
     \end{aligned}
$$
and
$$
\begin{aligned}
t_1 =&0.55161856837246093169757087608456543417211766450713\cr
       &\quad 8868116831699168666814224190486583439508658139692\color{gray}{6}\cr
       &\quad \color{gray}{63351381969733012016129364111250869850101334085360}\cr
       &\quad \color{gray}{70969237514708581622707399079704491867257671463809},
\end{aligned}
$$
which yields the estimate:
$$
\begin{aligned}
 &\dim_H(\mathcal{J}_{\mathcal T}) = 0.5516185683
7246093169
7570876084
5654341721
176
\cr
&\quad
6450718
86811683169916866681
4224190486
5834395086
58139692 \ldots.
\end{aligned}
$$

\end{example}

We next consider as a second example, $SG_3$, see Example \ref{ex:sg3}.

\begin{example}[$SG_3$ gasket] \label{ex:sg3_numerics}
With the same method as in the previous example, we estimate bounds on
$\dim_H({\mathcal J}_{SG_3})$ to $60$ decimal places:
$$
\begin{aligned}
t_0 =&
0.6175063018623522290424948743164070963419768663609616\cr
&\qquad 0395161\color{gray}{40619156598666691050499356772905041875773}\cr
 \end{aligned}
$$
and
$$
\begin{aligned}
t_1 =&
0.6175063018623522290424948743164070963419768663609616\cr
&\qquad 0395161\color{gray}{51934758805391761943498290334758478481658},\cr
\end{aligned}
$$
which yields the estimate:
$$
\begin{aligned}
\dim_H({\mathcal J}_{SG_3}) &=
0.61750
63018
62352
22904
24948
74316
\cr
&\qquad
40709
63419
76866
36096
16039
5161
 \ldots
\cr
\end{aligned}
$$
\end{example}

\begin{remark}
A significant contribution to  the time complexity of the algorithm is that  of estimating the top eigenvalue and corresponding eigenvector of an $m \times m$ matrix which is $O(n \cdot m^2)$
with $n$ denoting the number of steps of the power iteration method. Moreover, by perturbation theory one might expect that in order to get an error in the eigenvector of $\epsilon > 0$ one needs to choose $m = O(\log(1/\epsilon))$
and $n = O(\log(1/\epsilon))$. 
\end{remark}

\section{Conclusion}

In this note, we have leveraged the existing theory on Laplacians associated
to Sierpi\'nski lattices, infinite Sierpi\'nski gaskets and other post-critically
finite self-similar sets, in order to establish the Hausdorff dimensions of their respective spectra.
We used the insight that, by virtue of the iterative description of these spectra,
these dimensions coincide with those of the Julia sets of certain rational functions.
Since the contractive local inverse branches of these functions are non-linear,
the values of the Hausdorff dimensions are not available in an explicit closed form,
in contrast to the dimensions of the (infinite) Sierpi\'nski
gaskets themselves, or other self-similar fractals constructing using contracting similarities and satisfying
an open set condition.
Therefore we use the fact that the Hausdorff dimension can
be expressed implicitly as the unique zero of a so-called pressure function, which itself corresponds to the maximal positive
simple eigenvalue of a family of positive transfer operators. Using a min-max method combined with the Lagrange-Chebyshev
interpolation scheme we can rigorously estimate the leading eigenvalues for every operator in this family.
Combined with a bisection method we then accurately and efficiently estimate the zeros of the respective pressure functions, yielding
rigorous and effective bounds on the Hausdorff dimensions of the spectra of the relevant Laplacians.

\Addresses

\end{document}